\documentclass[11pt]{article}
\usepackage{amsmath, amssymb, amsthm}
\usepackage[margin=1in]{geometry}

\usepackage{parskip} 
\usepackage[colorlinks=true, linkcolor=magenta, citecolor=magenta, urlcolor=magenta]{hyperref}
\usepackage{graphicx}
\usepackage{tikz}

\newtheorem{theorem}{Theorem}[section]
\newtheorem{lemma}{Lemma}[section]

\theoremstyle{remark}

\newcommand{\authorinfo}{%
  \par\vspace{2em}
  \noindent
  Yulia Alexandr\\
  School of Mathematics and Statistics \\
  University of Melbourne\\
  \href{mailto:yulia.alexandr@unimelb.edu.au}
       {\texttt{yulia.alexandr@unimelb.edu.au}}
}

\title{Boundaries of logarithmic Voronoi cells can be transcendental}
\author{Yulia Alexandr}
\date{}

\begin{document}

\setlength{\abovedisplayskip}{7pt}
\setlength{\belowdisplayskip}{7pt}
\setlength{\abovedisplayshortskip}{5pt}
\setlength{\belowdisplayshortskip}{5pt}

\maketitle

\begin{abstract}
We show that logarithmic Voronoi boundaries for one-dimensional algebraic models in the probability simplex need not be algebraic. We give two explicit examples. For a union of two lines, the logarithmic Voronoi cell at a specific smooth rational point has a transcendental boundary point. For a cuspidal curve, the cell at the singular point contains a non-algebraic real-analytic boundary branch.
\end{abstract}

\section{Introduction}

Maximum likelihood estimation assigns to observed data the point or points on a statistical model that best explain it. Fixing a model point $p$, one can ask for all data vectors whose maximum likelihood estimate is $p$. This region in data space is the logarithmic Voronoi cell at $p$.

These cells describe what can be inferred about the data from a maximum likelihood estimate. If only the estimate $p$ is released, then every empirical distribution in $\operatorname{logVor}_{\mathcal M}(p)$ is compatible with that release. This connects logarithmic Voronoi cells to statistical disclosure limitation, where one seeks to understand what can be inferred about confidential data from released estimates \cite{slavkovic2023statistical}. It also motivates studying cells at singular and boundary points of the model, where the geometry of logarithmic Voronoi cells can change dramatically and the usual smooth critical-point methods no longer apply directly. The boundary of a cell captures a complementary feature of the estimation problem: it marks where the optimality of $p$ can be lost under perturbations of the data. In the examples studied here, such boundary points arise when $p$ ties with another point of the model.

Maximizing the likelihood is equivalent to minimizing the Kullback--Leibler divergence from the data to the model. Logarithmic Voronoi cells are therefore the Kullback--Leibler analogues of~Voronoi cells. Their study connects information geometry \cite{alexandr2025maximum, amari2000methods}, algebraic statistics \cite{drton2009lectures, sullivant2018algebraic}, and metric algebraic geometry \cite{breiding2024metric}. For Euclidean distance, Voronoi cells of real algebraic varieties are convex semialgebraic sets, and their algebraic boundaries reflect the geometry of the underlying variety \cite{cifuentes2022voronoi}. The likelihood setting differs in an essential way: the model is algebraic, but the objective function contains logarithms. As a result, these boundaries do not fall directly within the algebraic framework used for Voronoi cells of varieties. Related transcendence phenomena have been observed in maximum likelihood estimation for Gaussian mixture models \cite{amendola2016maximum}. In this paper, we show that even for one-dimensional algebraic models, logarithmic Voronoi boundaries can be~transcendental.

Logarithmic Voronoi cells were introduced in \cite{alexandr2021logarithmic}, where they were shown to be convex. For finite models, models of maximum likelihood degree one, linear models, and toric models, the cells are polytopes. Their combinatorial structure for discrete linear models was studied in \cite{alexandr2024logarithmic}; log-normal polytopes for squared linear models were studied in \cite{friedman2026squared}. The Gaussian analogue was developed in \cite{alexandr2024gaussian}. In that setting, the cells of the bivariate correlation model were shown to be semialgebraic, whereas those of unrestricted correlation models were conjectured to be transcendental.

\subsection*{Main results}

We show that logarithmic Voronoi boundaries need not be algebraic. In both examples, the model and the point whose cell we study are defined over $\mathbb Q$, so the transcendence appears in the data-space boundary rather than being inherited from the model point itself. For a union of two lines in~$\Delta_2$, we exhibit a smooth rational point $q$ such that $\operatorname{logVor}_{\mathcal M}(q)$ has a transcendental endpoint. The boundary condition reduces to a one-variable equation; a $3$-adic argument excludes rational solutions, while the Gelfond--Schneider theorem excludes algebraic irrational solutions.

For a cuspidal curve, we study the logarithmic Voronoi cell at its singular point $p_0$. We construct a real-analytic family of data for which $p_0$ ties with a smooth point of the curve. As the competing point approaches the boundary of the simplex, the corresponding branch acquires a logarithmic asymptotic term. The Newton--Puiseux theorem then shows that the branch is not algebraic, and a global maximization argument proves that it lies in $\partial\operatorname{logVor}_{\mathcal M}(p_0)$.

The geometry of the two examples is illustrated in Figure~\ref{fig:examples}.

\begin{figure}[ht]
    \centering
    \includegraphics[width=0.48\linewidth]{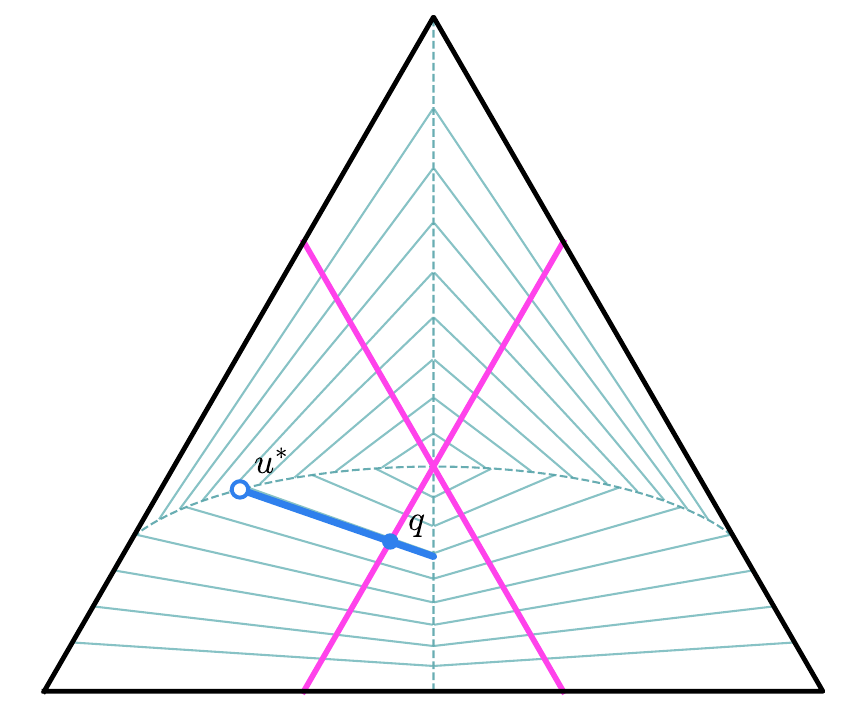}
    \hfill
    \includegraphics[width=0.48\linewidth]{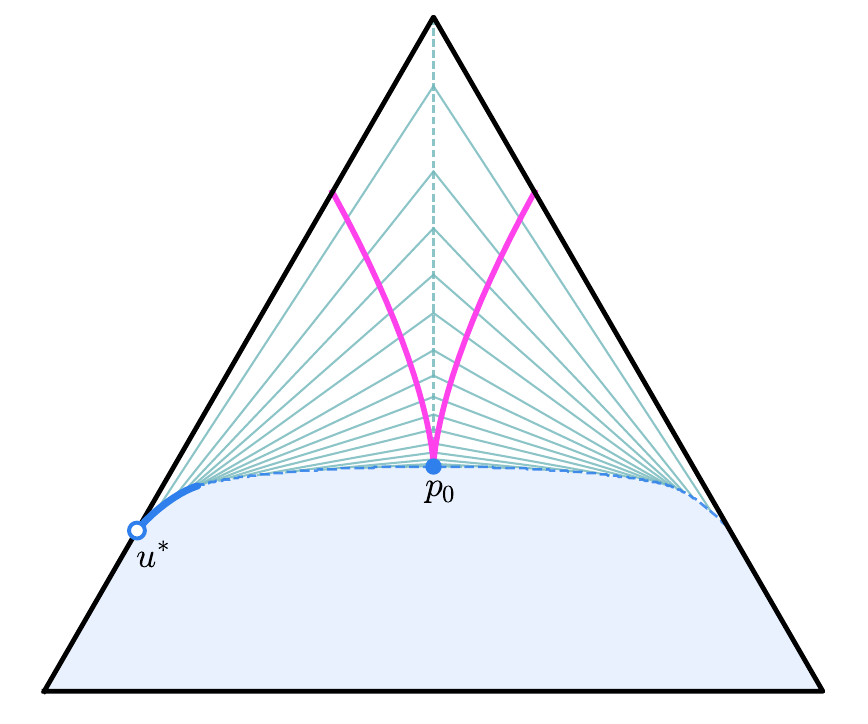}
    \caption{\small
    Logarithmic Voronoi decompositions for the union of two lines (left) and the cuspidal curve~(right). The models are shown in magenta and sampled logarithmic Voronoi cells in teal. Blue indicates the cell or boundary branch considered in each example, and $u^\ast$ marks the corresponding boundary point.}
    \label{fig:examples}
\end{figure}
\section{Preliminaries}

Let
$$
\Delta_{n-1}
=
\left\{
u\in\mathbb R_{\geq 0}^n:
u_1+\cdots+u_n=1
\right\}
$$
be the probability simplex. An \textit{algebraic statistical model} is a set
$\mathcal M=\mathcal V\cap \Delta_{n-1}$, where $\mathcal V\subseteq \mathbb R^n$
is a real algebraic variety. For data $u\in\Delta_{n-1}$, the \textit{log-likelihood function} with respect to $u$ is $\ell_u:\Delta_{n-1}\to \mathbb R\cup\{-\infty\}$ defined by
$\ell_u(p)=\sum_{i=1}^n u_i\log p_i$ with the conventions $0\log 0=0$ and $u_i\log 0=-\infty$ if $u_i>0$.
The value $\ell_u(p)$ measures how well the probability distribution $p \in \mathcal M$ explains the empirical distribution $u$. The \textit{maximum likelihood estimation problem} on a model $\mathcal M\subseteq\Delta_{n-1}$ is to maximize $\ell_u(p)$ over $p \in \mathcal{M}$. For $p\in\mathcal M$, the \textit{logarithmic Voronoi cell} at $p$ is
$$
\operatorname{logVor}_{\mathcal M}(p)
=
\left\{
u\in\Delta_{n-1}:
\ell_u(p)\geq \ell_u(q)\text{ for all }q\in\mathcal M
\right\}.
$$
Equivalently, $\operatorname{logVor}_{\mathcal M}(p)$ is the set of data vectors
for which $p$ is a global maximum likelihood estimate on $\mathcal M$. 

Let $p\in\mathcal M$ be a smooth point with positive coordinates and let $N_p\mathcal M$ denote the Euclidean normal space to $\mathcal M$ at $p$. The point $p \in \mathcal M$ is a critical point of $\ell_u$ if and only if $\nabla\ell_u(p) = (u_1/p_1,\ldots, u_n/p_n)\in N_p\mathcal M$. This condition defines the \textit{log-normal space} of $\mathcal M$ at $p$:
$$\log N_p\mathcal M=\{u\in\mathbb R^n:\nabla\ell_u(p)\in N_p\mathcal M\}.$$
The corresponding \textit{log-normal polytope} is
$\log\operatorname{Poly}_{\mathcal M}(p)=\log N_p\mathcal M\cap\Delta_{n-1}$. Thus, the log-normal polytope at $p$ is the set of data points in the simplex for which $p$ is a likelihood critical point.

For $p\in\mathcal M$, the notation $\partial\operatorname{logVor}_{\mathcal M}(p)$ denotes the boundary of $\operatorname{logVor}_{\mathcal M}(p)$ relative to its affine span. We use the word \textit{transcendental} in two related senses. For isolated boundary points, this is an arithmetic notion: a point $u\in\Delta_{n-1}$ is called \textit{algebraic} if all of its coordinates are algebraic numbers, and \textit{transcendental} otherwise. This is the sense used for the union of two lines. There, the logarithmic Voronoi cell is a line segment, so its boundary consists of endpoints, and we show that one endpoint is transcendental even though the fixed model point is rational.

For positive-dimensional boundary pieces, we use transcendence in a geometric sense. More precisely, we call a real-analytic piece of $\partial\operatorname{logVor}_{\mathcal M}(p)$ \textit{transcendental} if it is not contained in any proper real algebraic variety in the affine span of $\operatorname{logVor}_{\mathcal M}(p)$. This is the sense used for the cuspidal example, where the boundary contains a real-analytic branch that is not contained in any proper real algebraic variety.

\section{The union of two lines}
In this section we give the first example: a rational point on a reducible linear model whose logarithmic Voronoi cell has a transcendental endpoint. Let
$$
L_1=
\left\{
\left(\tfrac13+t,\tfrac13,\tfrac13-t\right):
-\tfrac13\leq t\leq \tfrac13
\right\},
\qquad
L_2=
\left\{
\left(\tfrac13,\tfrac13+s,\tfrac13-s\right):
-\tfrac13\leq s\leq \tfrac13
\right\},
$$
and let $\mathcal M=L_1\cup L_2$. We study the logarithmic Voronoi cell of the rational point
$$
q=\left(\tfrac49,\tfrac13,\tfrac29\right)\in L_1,
$$
given by the parameter $t = \frac 19$.

\begin{theorem}\label{thm:two-lines}
The logarithmic Voronoi cell $\operatorname{logVor}_{\mathcal M}(q)$ has an endpoint of the form
$$
u=\frac{(2,\beta,1)}{3+\beta},
$$
where $\beta\in(0,1)$ is transcendental. In particular, $\partial\operatorname{logVor}_{\mathcal M}(q)$ contains a transcendental endpoint.
\end{theorem}

We will use the following standard form of the Gelfond--Schneider theorem \cite[Theorem~10.1]{niven1956gelfond}.

\begin{theorem}[Gelfond--Schneider]\label{thm:gelfond-schneider}
If $a$ and $b$ are algebraic numbers with $a\notin\{0,1\}$ and $b$ irrational, then any value of $a^b$ is transcendental.
\end{theorem}

\begin{proof}[Proof of Theorem \ref{thm:two-lines}]

We proceed in three steps. First we identify the data points whose MLE on~$L_1$ is the fixed rational point $q$. Second, we compare the log-likelihood at $q$ with the maximum on $L_2$. Third, we show that the resulting boundary point is transcendental.

It is enough to work in the positive part of the simplex, since the endpoint we construct has positive coordinates. Write a data vector in ratio coordinates as
$$
u=\frac{(\alpha,\beta,1)}{\alpha+\beta+1},
$$
where $\alpha=u_1/u_3, \beta=u_2/u_3 >0$. This gives a reparametrization of the open simplex $\Delta^\circ$, since every point $u=(u_1,u_2,u_3)\in \Delta^\circ$ has $u_3>0$ and can therefore be written uniquely in this form. We first determine when the maximum likelihood estimate on $L_1$ is the fixed point~$q$. Along $L_1$, write $p(t)=\left(\tfrac13+t,\tfrac13,\tfrac13-t\right).$
Then
\[
\ell_u(p(t))=u_1\log\!\left(\tfrac13+t\right)+u_2\log\!\left(\tfrac13\right)+u_3\log\!\left(\tfrac13-t\right),
\]
hence the likelihood equation is
\[
\frac{d}{dt}\ell_u(p(t))=\frac{u_1}{\tfrac13+t}-\frac{u_3}{\tfrac13-t} = 0 \quad \implies \quad t=\frac{u_1-u_3}{3(u_1+u_3)}=\frac{\alpha-1}{3(\alpha+1)}.
\]

So the critical point of $\ell_u$ on $L_1$ is
\[
p^{(1)}(\alpha)=\left(\frac{2\alpha}{3(\alpha+1)},\frac13,\frac{2}{3(\alpha+1)}\right).
\]
Since $\ell_u$ is strictly concave on $L_1$, $p^{(1)}(\alpha)$ is the unique maximizer on~$L_1$. Moreover, $p^{(1)}(\alpha)=q$ if and only if $\alpha=2$. Thus the log-Voronoi cell at $q$ with respect to the model $L_1$ is the line segment
\begin{align}\label{eq:log-vor-wrt-L1}
\operatorname{logVor}_{L_1}(q) = 
\overline{\left\{\frac{(2,\beta,1)}{3+\beta}:\beta>0\right\}} \subseteq \Delta_2.
\end{align}
In particular, on this cell, we have $u_1 = 2 u_3$ and $u_2 = \beta u_3$. We now restrict to this cell and compare~$q$ with the maximum on $L_2$. Along $L_2$, write
$p(s)=\left(\tfrac13,\tfrac13+s,\tfrac13-s\right).$ Then
\[
\ell_u(p(s))=u_1\log\!\left(\tfrac13\right)+u_2\log\!\left(\tfrac13+s\right)+u_3\log\!\left(\tfrac13-s\right),
\]
giving the likelihood equation
\[
\frac{d}{ds}\ell_u(p(s))=\frac{u_2}{\tfrac13+s}-\frac{u_3}{\tfrac13-s} \quad \implies \quad
s=\frac{u_2-u_3}{3(u_2+u_3)}=\frac{\beta-1}{3(\beta+1)},
\]
and hence the critical point on $L_2$ is
\[
p^{(2)}(\beta)=\left(\frac13,\frac{2\beta}{3(\beta+1)},\frac{2}{3(\beta+1)}\right).
\]

Since $\ell_u$ is strictly concave on $L_2$, $p^{(2)}(\beta)$ is the unique maximizer on~$L_2$. Evaluating $\ell_u$ on $q$ and $p^{(2)}(\beta)$, while restricting $u$ to the data points in (\ref{eq:log-vor-wrt-L1}), we obtain
\[
\ell_u(q)=\ell_u\bigl(p^{(2)}(\beta)\bigr)
\quad\Longleftrightarrow\quad
K(\beta)=K(2),
\]
where
\[
K(r)=r\log\!\left(\frac{2r}{r+1}\right)+\log\!\left(\frac{2}{r+1}\right).
\]
Since
\[
K'(r)=\log\!\left(\frac{2r}{r+1}\right),
\]
the function $K$ is strictly decreasing on $(0,1)$ and strictly increasing on $(1,\infty)$. Also,
\[
K(2)=\log\!\left(\frac{32}{27}\right)
\qquad\text{and}\qquad
\lim_{r\to 0^+}K(r) = \log 2 > K(2) > 0 = K(1).
\]
By the Intermediate Value Theorem, there is a unique $\beta^*\in(0,1)$ such that $K(\beta^*)=K(2)$. For
$u^* = \frac{(2,\beta^*,1)} {3+\beta^*},$
the point $q$ ties with $p^{(2)}(\beta^*)$, so
$u^*\in \partial \operatorname{logVor}_\mathcal M(q).$ By the monotonicity of $K$, the inequality $\ell_u(q)\geq \ell_u(p^{(2)}(\beta))$ holds for $\beta^*\leq \beta\leq 2$, so $u^*$ is an endpoint of $\operatorname{logVor}_{\mathcal M}(q)$.

It remains to show that $\beta^*$ is transcendental. Suppose first that $\beta^*=m/n\in\mathbb{Q}$, written in lowest terms with $0<m<n$. Exponentiating and simplifying the equation $K(\beta^*)=\log(32/27)$ gives
\begin{align}\label{eq:exponentiation}
3^{3n}m^m n^n=2^{4n-m}(m+n)^{m+n}.
\end{align}
Since $\beta^*=m/n$ is in lowest terms, we have $\gcd(m,n)=1$. Hence
\[
\gcd(m,m+n)=\gcd(m,n)=1
\qquad\text{and}\qquad
\gcd(n,m+n)=\gcd(n,m)=1,
\]
so $m$, $n$, and $m+n$ are pairwise coprime. 
Thus, let $p$ be any prime dividing $m$ or $n$. Then $p$ divides the left-hand side of~(\ref{eq:exponentiation}), and hence also the right-hand side
$
2^{4n-m}(m+n)^{m+n}.
$
Since $p$ cannot divide $m+n$, it follows that $p=2$. Thus $m$ and $n$ are powers of $2$. Since $\gcd(m,n)=1$, one of them is $1$, and since $0<m<n$, we must have $m=1$. Then
\begin{align*}
    3^{3n}n^n=2^{4n-1}(n+1)^{n+1}.
\end{align*}

Taking $3$-adic valuations gives
\[
3n+n\,v_3(n)=(n+1)v_3(n+1).
\]
If $3\mid n$, then $v_3(n)\ge 1$ while $v_3(n+1)=0$, a contradiction. Hence $3\nmid n$, so $v_3(n)=0$, and therefore
\[
3n=(n+1)v_3(n+1).
\] Thus $n+1\mid 3n$. Since $\gcd(n,n+1)=1$, it follows that $n+1\mid 3$. Because $n>1$, we must have $n+1=3$, hence $n=2$. Substituting $(m,n)=(1,2)$ into (\ref{eq:exponentiation}) gives a contradiction. So $\beta^*\notin\mathbb{Q}$.

Now suppose that $\beta^*$ is algebraic irrational. Rewriting $K(\beta^*)=\log(32/27)$ gives
\begin{align}\label{eq:alg-irrarional}
\left(\frac{2\beta^*}{\beta^*+1}\right)^{\beta^*}=\frac{16(\beta^*+1)}{27}.
\end{align}

If $\beta^*$ were algebraic, then the base $\frac{2\beta^*}{\beta^*+1}$ would also be algebraic. Moreover, we have
$\frac{2\beta^*}{\beta^*+1}\neq 0,1.$
Indeed, it is not $0$ because $\beta^*>0$, and it is not $1$ because that would force
$2\beta^*=\beta^*+1,$
hence $\beta^*=1$, contrary to $\beta^*\in(0,1)$. Since $\beta^*$ is algebraic irrational, Theorem \ref{thm:gelfond-schneider} implies that the left-hand side of (\ref{eq:alg-irrarional})
must be transcendental. But the right-hand side of (\ref{eq:alg-irrarional}) is algebraic by assumption, a contradiction.

Hence $\beta^*$ is transcendental, and therefore $u^*=\frac{(2,\beta^*,1)}{3+\beta^*}$
is a transcendental boundary point of $\operatorname{logVor}_\mathcal M(q)$, which maps to an algebraic point on the model~$q$.

\end{proof}

\section{The cusp}

We now give a second example, this time at a singular point. Let
$$
\mathcal M
=
\left\{
(x,y,z)\in\Delta_2:
\left(z-\tfrac13\right)^3=(x-y)^2
\right\},
$$
and let $p_0=(\tfrac13,\tfrac13,\tfrac13)$ be its singular point. We will construct a real-analytic branch of $\partial\operatorname{logVor}_{\mathcal M}(p_0)$ and show that it is not algebraic. This branch consists of data vectors for which $p_0$ ties with a smooth point of $\mathcal M$. We use the parametrization
$$
p(t)
=\bigl(x(t),y(t),z(t)\bigr) =
\left(
\tfrac13-\tfrac12t^2+\tfrac12t^3,\,
\tfrac13-\tfrac12t^2-\tfrac12t^3,\,
\tfrac13+t^2
\right).
$$
Then $p(0)=p_0$, $x(t)+y(t)+z(t)=1$, and
$\left(z(t)-\tfrac13\right)^3=t^6=(x(t)-y(t))^2$.
Conversely, let $(x,y,z)\in\mathcal M$, and let $t\in\mathbb R$ be
defined by $t^3=x-y$. The defining equation of $\mathcal M$ gives
$z-\tfrac13=t^2$. Together with $x+y=1-z$, this gives
$(x,y,z)=p(t)$. Thus $p$ parametrizes all real points of $\mathcal M$.
Moreover, $p(t)$ is smooth for $t\neq 0$.

Let $r$ be the unique real root of $y(t)$. Indeed,
$y'(t)=-t(1+\tfrac32t)$, and
$y(-\tfrac23)=\tfrac7{27}$ and $y(0)=\tfrac13$ are both positive, so
$y$ has only one real root. Since $y(0)>0$ and $y(1)<0$, we have
$r\in(0,1)$. Since $x(t)=y(-t)$ and $z(t)>0$ for all $t$, it follows that
$\mathcal M\cap\Delta_2^\circ=\{p(t):-r<t<r\}$.

\begin{theorem}\label{thm:cusp}
The boundary $\partial\operatorname{logVor}_{\mathcal M}(p_0)$ contains
a real-analytic branch that is not contained in any proper real
algebraic subset of the affine plane $u_1+u_2+u_3=1$.
\end{theorem}

We will use the following standard consequence of Rouché's theorem
\cite[Proposition~17.5]{taylor2019complex}.

\begin{lemma}\label{lem:persistence}
Let $Q_t$ be a family of real polynomials of fixed degree whose
coefficients converge to those of a polynomial $Q^*$ as $t\to r$.
If all roots of $Q^*$ are simple, then, for $t$ sufficiently close to~$r$, each root of $Q^*$ has a neighborhood containing exactly one root
of $Q_t$. In particular, $Q_t$ and $Q^*$ have the same number of real
roots.
\end{lemma}

\begin{proof}
Choose pairwise disjoint closed disks around the roots of $Q^*$, each
containing exactly one root. For each real root, choose a disk centered
at that root, and hence invariant under complex conjugation. For each
nonreal root, choose a disk disjoint from the real axis. For each such disk~$D$, the polynomial $Q^*$ has no zeros on
$\partial D$, so
$$
m_D=\min_{z\in\partial D}|Q^*(z)|>0.
$$
Coefficientwise convergence implies that $Q_t\to Q^*$ uniformly on
$\partial D$. Thus, for $t$ sufficiently close to $r$, we have
$$
|Q_t(z)-Q^*(z)|<m_D\leq |Q^*(z)|
\qquad\text{for all }z\in\partial D.
$$
By Rouché's theorem, $Q_t$ and $Q^*$ have the same number of roots in
$D$, counted with multiplicity. Since the root of $Q^*$ in $D$ is
simple, $Q_t$ has exactly one root in $D$.

If $D$ is centered at a real root and its unique root of $Q_t$ were
nonreal, then its conjugate would also be a root of $Q_t$ in $D$,
contradicting uniqueness. Hence this root is real. If $D$ is centered
at a nonreal root, then its root is nonreal because $D$ is disjoint
from the real axis.
\end{proof}

\begin{proof}[Proof of Theorem~\ref{thm:cusp}]
We divide the proof into two parts. We first construct a non-algebraic
branch along which $p_0$ ties with a smooth critical point. We then
show that these two points are global maximizers.

\medskip
\noindent\textit{Construction of the tie branch.}
For $u=(u_1,u_2,u_3)\in\Delta_2^\circ$, set
$$
L_u(t)=\ell_u(p(t))
=u_1\log x(t)+u_2\log y(t)+u_3\log z(t).
$$
We seek data vectors for which $p_0=p(0)$ ties with a smooth critical
point $p(t)$. For $t\in(0,r)$, these conditions are
$L_u(t)=L_u(0)$ and $L_u'(t)=0$. Since $L_u(0)=\log(\tfrac13)$, set
$$
A(t)=\log(3x(t)),\qquad
B(t)=\log(3y(t)),\qquad
C(t)=\log(3z(t)),
$$
and
$$
\alpha(t)=\frac{x'(t)}{x(t)},\qquad
\beta(t)=\frac{y'(t)}{y(t)},\qquad
\gamma(t)=\frac{z'(t)}{z(t)}.
$$
Suppressing the dependence on $t$, the tie condition, the critical
point condition, and the simplex condition, respectively, become
\begin{align}
u_1A+u_2B+u_3C=0,\qquad
u_1\alpha+u_2\beta+u_3\gamma=0,\qquad
u_1+u_2+u_3=1.
\label{eq:cusp-system}
\end{align}
Whenever
\begin{align}
D(t)=A(\beta-\gamma)+B(\gamma-\alpha)+C(\alpha-\beta)
\label{eq:cusp-D}
\end{align}
is nonzero, the system \eqref{eq:cusp-system} has the unique solution via Cramer's rule:
\begin{align}
u_1(t)=\frac{B\gamma-C\beta}{D},\qquad
u_2(t)=\frac{-A\gamma+C\alpha}{D},\qquad
u_3(t)=\frac{A\beta-B\alpha}{D}.
\label{eq:cusp-u}
\end{align}

We study these solutions as $t\to r^-$. Since $y(r)=0$, we have
$\tfrac13=\tfrac12r^2+\tfrac12r^3$, and therefore $x(r)=r^3$ and
$z(r)=1-r^3$. Set
$$
A_0=\log(3r^3),\qquad
C_0=\log\bigl(3(1-r^3)\bigr),\qquad
\alpha_0=\frac{3r-2}{2r^2},\qquad
\gamma_0=\frac{2r}{1-r^3},
$$
and let $c=C_0-A_0$ and $K=-A_0\gamma_0+C_0\alpha_0$.

Since $r\in(0,1)$ and $2=3r^2+3r^3>6r^3$, we have
$r^3<\tfrac13$. Thus $A_0<0<C_0$, and in particular $c>0$. We will
also need $K>0$. Since $y(0.63)>0$ and $y(0.64)<0$, we have
$r\in(0.63,0.64)$. On this interval,
$-A_0>0.24$, $\gamma_0>1.68$, $C_0<0.82$, and
$|\alpha_0|<0.14$. Hence
$$
K=(-A_0)\gamma_0+C_0\alpha_0
\geq(-A_0)\gamma_0-C_0|\alpha_0|
>0.24\cdot1.68-0.82\cdot0.14>0.
$$

Write $s=r-t$, so that $s\to0^+$ as $t\to r^-$. Since $y(r)=0$ and
$y'(r)<0$, Taylor expansion at $r$ gives
$y(t)=as+O(s^2)$ and $y'(t)=-a+O(s)$, where
$a=-y'(r)>0$. It follows that
\begin{align}
B(t)=\log(3a)+\log s+O(s),\qquad
\beta(t)=-\frac1s+O(1).
\label{eq:cusp-B-beta}
\end{align}
The remaining functions are analytic at $r$, so
\begin{align}
A(t)=A_0+O(s),\qquad C(t)=C_0+O(s),\qquad
\alpha(t)=\alpha_0+O(s),\qquad
\gamma(t)=\gamma_0+O(s).
\label{eq:cusp-regular}
\end{align}
Substituting \eqref{eq:cusp-B-beta} and \eqref{eq:cusp-regular} into
\eqref{eq:cusp-D} gives
\begin{align}
D(t)=\frac{c}{s}+(\gamma_0-\alpha_0)\log s+O(1).
\label{eq:cusp-D-asymptotic}
\end{align}
In particular, \eqref{eq:cusp-D-asymptotic} shows that $D(t)\neq0$ for
$t<r$ sufficiently close to $r$. Writing
$d=\gamma_0-\alpha_0$, we get
\begin{align}
\frac1{D(t)}
=\frac{s}{c}-\frac{d}{c^2}s^2\log s+O(s^2).
\label{eq:cusp-D-inverse}
\end{align}

The numerators in \eqref{eq:cusp-u} satisfy
$-A\gamma+C\alpha=K+O(s)$ and
$B\gamma-C\beta=C_0/s+\gamma_0\log s+O(1)$. Therefore,
using \eqref{eq:cusp-D-inverse},
\begin{align}
u_2(t)=\frac{K}{c}s+O\bigl(s^2|\log s|\bigr),
\qquad
u_1(t)=\frac{C_0}{c}+\frac{K}{c^2}s\log s+O(s).
\label{eq:cusp-u-asymptotics}
\end{align}
Since $u_3=1-u_1-u_2$, it follows that
$$
u_1(t)\longrightarrow\frac{C_0}{c}>0,\qquad
u_2(t)\longrightarrow0^+,\qquad
u_3(t)\longrightarrow\frac{-A_0}{c}>0.
$$
Thus $u(t)\in\Delta_2^\circ$ for $t<r$ sufficiently close to $r$.
By construction, $p_0$ and $p(t)$ have the same log-likelihood with
respect to $u(t)$, and $p(t)$ is a critical point of
$\ell_{u(t)}$ on $\mathcal M$.

Let $u_1^*=C_0/c$. Since $K/c>0$, \eqref{eq:cusp-u-asymptotics} gives
\begin{align}
s=\frac{c}{K}u_2+O\bigl(u_2^2|\log u_2|\bigr).
\label{eq:cusp-s-u2}
\end{align}
Substituting \eqref{eq:cusp-s-u2} into the expression for $u_1$ in
\eqref{eq:cusp-u-asymptotics} gives
\begin{align}
u_1-u_1^*=\frac1c\,u_2\log u_2+O(u_2)
\qquad\text{as }u_2\to0^+.
\label{eq:cusp-branch-asymptotic}
\end{align}
Since $D(t)\neq0$, the formulas in \eqref{eq:cusp-u} define a
real-analytic map $t\mapsto u(t)$ on an interval $(r-\delta,r)$.

We now show that its image is not algebraic. Since
$u_2(t)=\frac{K}{c}(r-t)+O((r-t)^2|\log(r-t)|)$, the coordinate $u_2$
is a local parameter near the limiting point $(u_1^*,0)$. If the branch
were contained in a proper real algebraic subset of the data plane,
then it would be contained in the zero set of some nonzero polynomial
in $u_1$ and $u_2$. Since $u_2$ is a local parameter on the branch,
the Newton--Puiseux theorem \cite[Theorem~VII.7]{flajolet2009analytic}
would then give a locally convergent expansion of $u_1-u_1^*$ in
fractional powers of $u_2$,
\begin{align}
u_1-u_1^*=\sum_{m\geq m_0}a_m u_2^{m/N}
\label{eq:cusp-puiseux}
\end{align}
for some $N\geq1$. After dividing \eqref{eq:cusp-puiseux} by $u_2$,
the leading behavior would be a power of $u_2$ or a finite limit. On
the other hand, dividing \eqref{eq:cusp-branch-asymptotic} by $u_2$
gives
$$
\frac{u_1-u_1^*}{u_2}
=\frac1c\log u_2+O(1),
$$
which diverges logarithmically as $u_2\to0^+$. This is impossible for
a Puiseux series, so the branch is not algebraic.

\medskip
\noindent\textit{Global maximality.}
It remains to show that this branch lies on the actual log-Voronoi
boundary. Since $u(t)\in\Delta_2^\circ$, the log-likelihood is $-\infty$
at $p(-r)$ and $p(r)$, where $x(-r)=0$ and $y(r)=0$, respectively.
It is therefore enough to study the likelihood on
$\mathcal M\cap\Delta_2^\circ=\{p(\tau):-r<\tau<r\}$. Fix $t<r$
sufficiently close to $r$, and define
$$
H_t(\tau)
=\ell_{u(t)}(p(\tau))-\ell_{u(t)}(p_0),
\qquad -r<\tau<r.
$$
We already know that $H_t(0)=H_t(t)=0$ and $H_t'(t)=0$. Differentiating gives
$$
H_t'(\tau)
=\frac{\tau}{2x(\tau)y(\tau)z(\tau)}P_t(\tau),
$$
where
$$
P_t(\tau)
=u_1(t)(3\tau-2)y(\tau)z(\tau)
-u_2(t)(3\tau+2)x(\tau)z(\tau)
+4u_3(t)x(\tau)y(\tau).
$$
The polynomial $P_t$ has degree six. Indeed, its leading coefficient is
$-\tfrac32u_1(t)-\tfrac32u_2(t)-u_3(t)
=\tfrac12u_3(t)-\tfrac32<0$.

Thus the smooth likelihood equation is a degree-six equation. In other words, the model $\mathcal{M}$ has ML degree 6 \cite{ml-degree}.  We will
show that, for the data $u(t)$ constructed above, only two of its six
roots are real. Together with the singular critical point
$\tau=0$, this will give exactly three real critical points of
$H_t$. As $t\to r^-$, we have
$$
u(t)\longrightarrow u^*=(u_1^*,0,u_3^*),
\qquad
u_1^*=\frac{C_0}{c},\qquad
u_3^*=\frac{-A_0}{c}.
$$
Hence $P_t$ converges coefficientwise to
$$
P^*(\tau)=y(\tau)M(\tau),
\qquad
M(\tau)=u_1^*(3\tau-2)z(\tau)+4u_3^*x(\tau).
$$

We first consider the roots of $y$. Since
$y'(\tau)=-\tau(1+\tfrac32\tau)$ and
$y(-\tfrac23)=\tfrac7{27}>0$, $y(0)=\tfrac13>0$, and $y(r)=0$,
the polynomial $y$ has exactly one real root, namely $r$. This root is
simple because $y'(r)<0$. Its two nonreal roots are also simple, since
a repeated nonreal root of a real cubic would force its conjugate to
be repeated as well.

Next, consider $M$. Its derivative is
$
M'(\tau)=(6+3u_1^*)\tau^2-4\tau+u_1^*.
$
We claim that $u_1^*>\tfrac23$. Indeed,
$$
u_1^*>\frac23
\quad\Longleftrightarrow\quad
27r^6(1-r^3)>1,
$$
and the estimate $r\in(0.63,0.64)$ gives
$27r^6(1-r^3)>27(0.63)^6(1-(0.64)^3)>1$.
It follows that the discriminant
$16-4(6+3u_1^*)u_1^*$ of $M'$ is negative. Hence $M$ is strictly
increasing. Moreover, $M(0)=(4-6u_1^*)/3<0$. At $\tau=r$,
$$
u_1^*\alpha_0+u_3^*\gamma_0
=\frac{K}{c}
=\frac{r}{2x(r)z(r)}M(r)>0,
$$
so $M(r)>0$. Thus $M$ has a unique real root
$\tau_{\min}\in(0,r)$. This root is simple because
$M'(\tau)>0$ for all $\tau\in\mathbb R$. Its two nonreal roots are also
simple, since otherwise their conjugates would also be repeated, which
is impossible for a cubic.

The polynomials $y$ and $M$ have no common root. They do not share a
real root because $M(r)>0$. To exclude a common nonreal root, let
$q(\tau)=\tau^3+\tau^2-\tfrac23$, so that
$y(\tau)=-\tfrac12q(\tau)$. Modulo $q$,~we~get
\begin{align}
M(\tau)\equiv
-(u_1^*+4)\tau^2+u_1^*\tau-\frac43u_1^*+\frac83.
\label{eq:cusp-M-remainder}
\end{align}
If $y$ and $M$ shared a nonreal root, then, together with its conjugate,
it would be a root of the quadratic factor
\begin{align}
\tau^2+(1+r)\tau+\frac{2}{3r}
\label{eq:cusp-q-factor}
\end{align}
of $q$. The factor in \eqref{eq:cusp-q-factor} would therefore divide
the quadratic remainder in \eqref{eq:cusp-M-remainder}. Comparing the
coefficients of $\tau^2$ and $\tau$ would give
$u_1^*=-(u_1^*+4)(1+r)$, which is impossible, since $u_1^*>0$ and~$r>0$.

Consequently, $P^*=yM$ has two real simple roots, namely $r$ and
$\tau_{\min}$, and four nonreal simple roots. By
Lemma~\ref{lem:persistence}, $P_t$ has exactly two real roots for
$t$ sufficiently close to $r$. One is $t$, since $H_t'(t)=0$ and
$t\to r$. Denote the other by $\tau_t$. Since
$\tau_t\to\tau_{\min}\in(0,r)$, we have
$$
0<\tau_t<t<r
$$
for $t$ sufficiently close to $r$. Thus the degree-six likelihood
equation has two real roots and four nonreal roots. Including the
singular critical point $\tau=0$, the only real critical points of
$H_t$ are $0$, $\tau_t$, and $t$. At the singular point,
$$
H_t''(0)=9u_3(t)-3.
$$
Since $u_3(t)\to u_3^*=1-u_1^*<\tfrac13$, we have $H_t''(0)<0$ for
$t$ sufficiently close to $r$. Thus $\tau=0$ is a strict local
maximum. Moreover,
$$
H_t(\tau)\longrightarrow-\infty
\quad\text{as }\tau\to-r^+\text{ or }\tau\to r^-,
$$
because $x(\tau)\to0$ as $\tau\to-r^+$ and $y(\tau)\to0$ as
$\tau\to r^-$.

Since the only real critical points of $H_t$ are
$0$, $\tau_t$, and $t$, the sign of $H_t'$ is constant on each of the
intervals $(-r,0)$, $(0,\tau_t)$, $(\tau_t,t)$, and $(t,r)$.
Because $H_t''(0)<0$, we have $H_t'>0$ immediately to the left of $0$
and $H_t'<0$ immediately to the right. Hence $H_t$ is strictly
increasing on $(-r,0)$ and strictly decreasing on $(0,\tau_t)$, so
$H_t(\tau_t)<H_t(0)=0$. Since $H_t(t)=0$, it follows that $H_t$ is
strictly increasing on $(\tau_t,t)$. Finally, since
$H_t(\tau)\to-\infty$ as $\tau\to r^-$, $H_t$ is strictly decreasing
on $(t,r)$.

\begin{center}
\begin{tikzpicture}[x=1.25cm,y=1.15cm]

\draw[->, gray!70] (-4.2,0) -- (4.25,0) node[right] {$\tau$};

\draw[thick, teal]
(-3.8,-2.5)
.. controls (-3.0,-1.0) and (-0.8,0) .. (0,0)
.. controls (0.4,0) and (0.65,-1.15) .. (1.35,-1.15)
.. controls (2.0,-1.15) and (2.4,0) .. (2.8,0)
.. controls (3.2,0) and (3.55,-1.0) .. (3.8,-2.5);

\fill[violet] (0,0) circle (1.8pt);
\fill[gray] (1.35,-1.15) circle (1.8pt);
\fill[violet] (2.8,0) circle (1.8pt);

\draw[dashed, magenta] (0,0) -- (0,-0.1);
\draw[dashed, gray] (1.35,-1.15) -- (1.35,0);
\draw[dashed, magenta] (2.8,0) -- (2.8,-0.1);

\draw (-3.8,0.07) -- (-3.8,-0.07) node[below] {$-r$};
\draw (0,0.07) -- (0,-0.07) node[below] {$0$};
\draw (1.35,0.07) -- (1.35,-0.07) node[below] {$\tau_t$};
\draw (2.8,0.07) -- (2.8,-0.07) node[below] {$t$};
\draw (3.8,0.07) -- (3.8,-0.07) node[below] {$r$};

\node[teal,left] at (-3.85,-2.4) {$-\infty$};
\node[teal,right] at (3.85,-2.4) {$-\infty$};

\node[violet,above left] at (0,0) {$H_t(0)=0$};
\node[gray,below] at (1.35,-1.15) {$H_t(\tau_t)$};
\node[violet,above] at (2.8,0) {$H_t(t)=0$};

\end{tikzpicture}

\small\emph{Schematic behavior of the function $H_t$.}
\end{center}
Therefore
$
H_t(\tau)\leq0
$
for every $\tau\in(-r,r)$, with equality if and only if
$\tau=0$ or $\tau=t$. Hence $p_0$ and $p(t)$ are the only global
maximizers of $\ell_{u(t)}$ on $\mathcal M$.

It remains only to check that $u(t)$ is a boundary point of the
log-Voronoi cell. Since $p_0$ and $p(t)$ are both global maximizers at
$u(t)$, we have $u(t)\in\operatorname{logVor}_{\mathcal M}(p_0)$.
For fixed $t$, let
$$
F_t(v)=\ell_v(p_0)-\ell_v(p(t)).
$$
By definition,
$\operatorname{logVor}_{\mathcal M}(p_0)\subseteq\{v:F_t(v)\geq0\}$,
while $F_t(u(t))=0$. On the other hand, $p_0$ itself belongs to
$\operatorname{logVor}_{\mathcal M}(p_0)$ and, since $p(t)\neq p_0$,
$
F_t(p_0)=\ell_{p_0}(p_0)-\ell_{p_0}(p(t))>0.
$
Thus $F_t$ is nonconstant on the affine span of
$\operatorname{logVor}_{\mathcal M}(p_0)$, and $F_t=0$ defines a
supporting hyperplane of the cell containing~$u(t)$. Therefore
$
u(t)\in\partial\operatorname{logVor}_{\mathcal M}(p_0).
$
\end{proof}

\section*{Acknowledgments}

I am grateful to Serkan Ho\c{s}ten for his helpful comments on this manuscript. I also thank Álvaro Ribot and Anna Seigal for useful discussions about Voronoi cells at singular points, which helped motivate my return to related questions in the logarithmic setting.

The author acknowledges the use of ChatGPT for language editing and exploratory discussion. All mathematical results were developed and verified by the author.

\bibliographystyle{plain}
\bibliography{references}

\vspace{-1.75em}
\authorinfo

\end{document}